\documentclass[epic,eepic,11pt]{amsart}

\usepackage{amsfonts,mathrsfs}
\usepackage[mathscr]{eucal}
\usepackage{mathtools}
\usepackage{amssymb}
\usepackage{amscd}
\usepackage{fancyhdr}
\usepackage[all,cmtip]{xy}

\usepackage{euler,eucal}

\DeclareMathAlphabet{\mathbf}{T1}{ppl}{bx}{n}
\DeclareMathAlphabet{\mathrm}{T1}{ppl}{m}{n}

\numberwithin{equation}{section}

\newcommand\note[1]%
{$^\dagger$\marginpar{\footnotesize{$^\dagger${#1}}}}

\def\({\left(}
\def\){\right)}
\def\<{\left<}
\def\>{\right>}

\newtheorem{theorem}{Theorem}[section]
\newtheorem{proposition}[theorem]{Proposition}
\newtheorem{lemma}[theorem]{Lemma}
\newtheorem{definition}[theorem]{Definition}

\newtheorem{corollary}[theorem]{Corollary}

\theoremstyle{definition}
\newtheorem{example}[theorem]{Example}
\newtheorem{remark}[theorem]{Remark}

\newcommand\bb[1]{{\text{\bf#1}}}

\newcommand\Z{\bb{Z}}

\newcommand\R{\mathbb{R}}

\newcommand     {\comment}[1]   {}
\newcommand{\mute}[2] {}
\newcommand     {\printname}[1] {}

\newcommand\funclim[1]{\operatorname*{\mathrm{#1}}}

\renewcommand\lim{\funclim{lim}}

\newcommand\sur{\mathrel{\to\kern-1.8ex\to}}
\newcommand\iso{\mathrel{\hookrightarrow\kern-1.8ex\to}}

\newcommand\longhookrightarrow{\lhook\joinrel\longrightarrow}

\newcommand\longsur{\mathrel{\longrightarrow\kern-1.8ex\to}}
\newcommand\longiso{\mathrel{\longhookrightarrow\kern-1.8ex\to}}

\begin{document}

\bibliographystyle{amsalpha}
\date{\today}

\title{Lefschetz contact manifolds and odd dimensional symplectic geometry}

\author{Yi Lin }

 \date{\today}
\begin{abstract} In the literature, there are two different versions of Hard Lefschetz theorems for a compact Sasakian manifold. The first version, due to Kacimi-Alaoui, asserts that the basic cohomology groups of a compact Sasakian manifold satisfies the transverse Lefschetz property. The second version, established far more recently by Cappelletti-Montano, De Nicola, and Yudin, holds for the De Rham cohomology groups of a compact Sasakian manifold. In the current paper, using the formalism of odd dimensional symplectic geometry, we prove a Hard Lefschetz theorem for compact $K$-contact manifolds, which implies immediately that the two existing versions of Hard Lefschetz theorems are mathematically equivalent to each other.

 Our method sheds new light on the Hard Lefschetz property of a Sasakian manifold. It enables us to  give a simple construction of simply-connected $K$-contact manifolds without any Sasakian structures in any dimension $\geq 9$, and answer an open question asked by Boyer and late Galicki concerning the existence of such examples. 


\end{abstract}
\maketitle



\setcounter{section}{0} \setcounter{subsection}{0}

\section{Introduction}

 Conceptually, Sasakian geometry can be thought as an odd-dimensional counterpart of K\"ahler geometry. It is naturally related to two K\"ahler geometries. On the one hand, the metric cones of Sasakian  manifolds must be K\"ahler. On the other hand, for any Sasakian manifold, the one dimensional foliation defined by the Characteristic Reeb vector field is transversally K\"ahler.

 In K\"ahler geometry, the Hard Lefschetz theorem is a remarkable classical result which has many important applications. In Sasakian geometry, it has been long known that the basic cohomology groups of any compact Sasakian manifold satisfy the Hard Lefschetz property. This result was due to El Kacimi-Alaoui \cite{ka90}, who derived it as an immediate consequence of the basic Hodge theory on basic forms. However, more recently, Cappelletti-Montano, De Nicola, and Yudin \cite{CNY13} also proved a Hard Lefschetz theorem which holds for the De Rham cohomology groups of a compact Sasakian manifold. Now that there are two different versions of the Hard Lefschetz theorem for a compact Sasakian manifold in the literature, one naturally wonder how they are conceptually related to each other.

%

 In K\"ahler geometry, taking cup product with the cohomology class of a given k\"ahler $2$-form naturally gives rise to Lefschetz maps between cohomology groups. In contrast, in Sasakian geometry, it is not obvious at all how to define Lefschetz maps appropriately. In \cite{CNY13}, the authors first proved the existence of Lefschetz maps by examining the spectral properties of the Laplacian associated to a compatible Sasakian metric very carefully, and then verified that these maps are independent of the choices of compatible Sasakian metrics.

More precisely, let $(M,\eta,g)$ be a $2n+1$ dimensional compact Sasakian manifold with a contact one form $\eta$. It is shown in \cite{CNY13} that for any $0\leq k\leq n$, the map
 \[ Lef_k: \Omega^{k}(M)\rightarrow \Omega^{2n+1-k}(M),\,\,\alpha\mapsto \eta\wedge (d\eta)^{n-k}\wedge \alpha\]
 sends harmonic forms to harmonic forms, and therefore induces a map
 \begin{equation}\label{Lef-map} Lef_k: H^{k}(M,\R)\rightarrow H^{2n+1-k}(M,\R)\end{equation}
 at the cohomology level.


In a different direction, Zhenqi He studied in his Ph.D thesis \cite{He10} the geometry of odd dimensional symplectic manifolds, which include contact manifolds as special examples. On any odd dimensional symplectic manifold, there is also a notion of a canonical Reeb vector field. When the odd dimensional symplectic structure is induced by a contact structure, the notion of a Reeb vector field in odd dimensional  symplectic geometry agrees with the usual one in contact geometry. Among other things, He developed symplectic Hodge theory on an odd dimensional symplectic manifold, which applies to the basic cohomology with respect to the Reeb vector field.

Inspired by the recent work of \cite{CNY13}, the author initiate in this paper a symplectic hodge theoretic approach to  the Hard Lefschetz theorem of Sasakian manifolds, and more generally, the Hard Lefschetz theorem of $K$-contact manifolds. Throughout this paper, a $K$-contact manifold $(M,\eta)$ is said to satisfy the {\bf transverse Hard Lefschetz property} if its basic cohomology (with respect to the canonical Reeb vector field induced by $\eta$) satisfies the Hard Lefschetz property. Among other things, we proved in this paper that a compact $K$-contact manifold satisfies the Hard Lefschetz property in the sense of \cite{CNY13},  if and only if it satisfies the transverse Hard Lefschetz property. As an immediate consequence, it implies that for a compact Sasakian manifold, the two different versions of Hard Lefschetz theorems  in the literature are mathematically equivalent to each other.

Our approach involves two important pieces of technology: a long exact sequence \cite[Sec.7.2]{BG08} which relates the basic cohomology to the De Rham cohomology of the $K$-contact manifold itself, and the odd dimensional symplectic Hodge theory developed in \cite{He10}. In this approach, the symplectic  Hodge theory replaces the role of the Riemannian Hodge theory used in \cite{CNY13}, and allows us to show that when the transverse Hard Lefschetz property holds for a compact $K$-contact manifold, then for any $0\leq k\leq n$, the Lefschetz map (\ref{Lef-map}) is well defined.



  In symplectic geometry, many examples of symplectic manifolds without any K\"ahler structures have been constructed. It has been very well understood that the category of symplectic manifolds is much larger than the category of K\"ahler manifolds. In contrast, much less is known about the differences between $K$-contact manifolds and Sasakian manifolds. Indeed, in the Open Problem 7.4 of their monograph \cite{BG08}, Boyer and late Galicki asked the question of whether there exist examples of simply-connected compact $K$-contact manifolds which do not support any Sasakian structures. In the present paper, as a concrete application of the Lefschetz property of a $K$-contact manifold,  we give a simple construction of simply-connected $K$-contact manifolds without any Sasakian structures in any dimension $\geq 9$.

In the literature, first examples of simply-connected $K$-contact manifolds appeared in \cite{HT13}, which was posted on arxiv a few months before our paper. Their methods depend on the theory of fat bundles developed by Sternberg, Weinstein, and Lerman, and yield examples in any dimension $\geq 9$, c.f. \cite[Thm. 5.4]{HT13}.
After the present paper was posted on arxiv, two works (\cite{BFMT14},\cite{MT15}) have appeared addressing Open Problem 7.4 in \cite{BG08}.  In \cite{BFMT14}, it is proved that all higher order Massey products for simply connected Sasakian manifolds vanish, although there are Sasakian manifolds with non-vanishing triple Massey products. This yields examples of simply connected K-contact non-Sasakian manifolds in dimensions $\geq17$. Combing the methods of homotopy theory and symplectic surgery,  in \cite{MT15} it is shown that there exist  seven dimensional simply-connected compact $K$-contact manifolds without any Sasakian structures.


This paper is organized as follows.  Section \ref{odd-sym-hodge} reviews the machinery developed in \cite{He10} on odd dimensional symplectic Hodge theory. Section \ref{review-contact} collects some facts from contact and Sasakian geometry we need in this paper. Section \ref{main-result} proves that a compact $K$-contact manifold satisfies the Hard Lefschetz property in the sense of \cite{CNY13} if and only if it satisfies the transverse Hard Lefschetz property. Section \ref{K-contact-examples} produces simply-connected examples of $K$-contact manifolds without any Sasakian structures. 


\subsection*{Acknowledgement} I am very grateful to Cappelletti-Montano, De Nicola, and Yudin for pointing out a mistake in an early version of our construction of simply-connected $K$-contact manifolds without any Sasakian structures. I would also like to thank R. Sjamaar and Z. Wang for their interests in this work.

\section{Review of odd dimensional symplectic Hodge theory }\label{odd-sym-hodge}

In this section we present a brief review of background materials in odd dimensional symplectic Hodge theory. We refer to \cite{He10}
for more details on odd dimensional version of symplectic Hodge theory, and to \cite{brylinski;differential-poisson} and \cite{Yan96}
for general background on symplectic Hodge theory.  We begin with the definition of odd-dimensional symplectic  manifolds.

\begin{definition} (\cite{He10})\label{odd-symplectic} Suppose that $M$ is a manifold of dimension $2n+1$ with a volume form $\Omega$ and a
closed $2$-form $\omega$ of maximum rank. Then the triple $(M,\omega, \Omega)$ is called an odd-dimensional symplectic manifold.  \end{definition}

\begin{example}
A contact manifold $M$ with a contact form $\eta$ naturally gives rise to an odd-dimensional symplectic manifold $(M, \omega, \Omega)$ with $\omega=d\eta$ and $\Omega=\eta\wedge \dfrac{(d\eta)^n}{n!}$.
\end{example}

Throughout the rest of this section, we assume that $(M,\omega, \Omega)$ is a $2n+1$ dimensional
symplectic manifold  as given in Definition \ref{odd-symplectic}. We observe that since $\omega$ is of maximum rank, $\text{ker}\,\omega$ is a one dimensional foliation on $M$; moreover, there is a canonical vector field $\xi$, called the Reeb vector field, given by
\[ \iota_{\xi}\omega=0,\,\,\,\iota_{\xi}\Omega=\dfrac{\omega^n}{n!}.\]

We define the space of horizontal and basic forms, as well as basic De Rham cohomology group  on $M$ as follows.
\begin{equation}\label{basic} \begin{split}   &  \Omega_{hor}(M)=\{\alpha \in \Omega(M)\vert \iota_{\xi}\alpha=0\},\\
 &\Omega_{bas}(M)=\{\alpha \in \Omega(M)\vert \iota_{\xi}\alpha=0,\, \mathcal{L}_{\xi}\alpha=0\},\\&
 H^k_B(M,\R)=\dfrac{\text{ker} d\cap \Omega^k_{bas}(M)}{\text{im} d(\Omega_{bas}^{k-1}(M))} .\end{split} \end{equation}

The closed two $2$-form $\omega$ induces a non-degenerate pairing $G(\cdot,\cdot)$ on $\Omega_{hor}^k(M)$, which gives rise to the the following definition of the symplectic Hodge star $\star$ on horizontal forms.
\[ \beta_k\wedge \star\alpha_k = G(\beta_k,\alpha_k)\dfrac{\omega^n}{n!},\]
where $\alpha_k, \beta_k\in \Omega^{k}_{hor}(M)$.

It is easy to check that the symplectic Hodge star operator maps basic forms to basic forms. So there is a symplectic Hodge star operator on the space of basic forms.
\[ \star: \Omega_{bas}^{k}(M)\rightarrow \Omega^{2n-k}_{bas}(M).\]

The symplectic Hodge operator gives rise to the following symplectic Hodge adjoint operator of the exterior differential $d$.
 \[ \delta \alpha_k=(-1)^{k+1}\star d\star\alpha_k,\,\,\,\alpha_k\in\Omega_{bas}^k(M).\]

In this context, a basic form $\alpha$ is said to be symplectic Harmonic if and only if $d\alpha=\delta\alpha=0$.

There are three important operators, the Lefschetz map $L$, the dual Lefschetz map $\Lambda$, and the degree counting map $H$ which are defined on basic forms as follows.

\begin{equation} \label{three-canonical-maps} \begin{aligned} & L :\Omega_{bas}^*(M) \rightarrow \Omega_{bas}^{*+2}(M), \,\,\,\alpha \mapsto \alpha \wedge \omega,\\
 & \Lambda: \Omega_{bas}^*(M) \rightarrow \Omega_{bas}^{*-2}(M),\,\,\,\alpha \mapsto \star L\star \alpha,\\
& H: \Omega^k_{bas}(M)\rightarrow \Omega^k_{bas}(M),\,\,\,H(\alpha)=(n-k)\alpha,\,\,\,\alpha \in \Omega^{k}_{bas}(M).\end{aligned}\end{equation}

The actions of $L$, $\Lambda$ and $H$ on $\Omega_{bas}(M)$ satisfy the following commutator relations.

\begin{equation} \label{sl2-module-on-forms} [  \Lambda, L]=H, \,\,\,[H, \Lambda]=2\Lambda,\,\,\,[H, L]=-2L.
\end{equation}

Therefore, these three operators define a representation of the Lie algebra $sl(2)$ on $\Omega(M)$. Although the $sl_2$-module
$\Omega(M)$ is infinite dimensional, there are only finitely many eigenvalues of the operator $H$. The $sl_2$-modules of this type
are studied in great details in \cite{Ma95} and \cite{Yan96}.  Among other things,
we have the following results.

\begin{lemma} \label{yan's-result}  Let $(M,\omega,\Omega)$ be a $2n+1$ dimensional symplectic manifold. For any $0\leq k\leq n$, $\alpha\in \Omega_{bas}^k(M)$ is said to be primitive if  $L^{n-k+1}\alpha=0$. Then we have that
\begin{enumerate}

 \item [a)]
 a basic $k$-form $\alpha$ is primitive if and only if $\Lambda \alpha=0$;

 \item[b)] any differential form
$\alpha_k \in \Omega_{bas}^k(M)$ admits a unique Lefschetz decomposition
\begin{equation}\label{lefschetz-decompose-forms} \alpha_k =\displaystyle \sum_{r\geq \text{max}(\frac{k-n}{2}, 0)} \dfrac{L^r}{r!}\beta_{k-2r},\end{equation}
where $\beta_{k-2r}$ is a primitive basic form of degree $k-2r$.

\end{enumerate}
\end{lemma}
\begin{remark}Throughout the rest of this paper, we will denote the space of primitive basic
 $k$-forms on $M$ by $\mathcal{P}_{bas}^k(M)$.

\end{remark}

\begin{definition} (\cite{He10})\label{odd-HLP}
The $2n+1$ dimensional symplectic manifold $M$ is said to satisfy the {\bf transverse Hard Lefschetz property} if and only if for any $0\leq k\leq n$, the Lefschetz map
  \begin{equation}\label{Lefschetz-map} L^{n-k}: H^{k}_B(M)\rightarrow H_B^{2n-k}(M)\, \,\,[\alpha]_B\mapsto [\omega^{n-k}\wedge \alpha]_B
  \end{equation} is an isomorphism.
 \end{definition}

\begin{remark}
We say that a one form $\eta\in \Omega(M)$ is a connection $1$-form if $\iota_{\xi}\eta=1$ and if $\mathcal{L}_{\xi}\eta=0$.
It is shown in \cite{He10} that if $M$ is compact, and if there is a connection $1$-form on $M$, then $\omega^k$ always represents a non-trivial cohomology class
in $H^{2k}_B(M,\R)$.  Clearly, if $M$ is a contact manifold with a contact one form $\eta$, then $\eta$ will be a connection $1$-form on $M$. If $M$ is also compact, then $\omega^k$ always represents a non-trivial cohomology class in $H^{2k}_B(M,\R)$.
\end{remark}

Among other things, \cite{He10} extended Mathieu's theorem, as well as the symplectic $d\delta$-lemma, to the odd dimensional case,

  \begin{theorem}\label{Mathieu's-theorem} ( \cite{Ma95}, \cite{He10}) On a compact odd dimensional symplectic manifold $M$, every basic De Rham cohomology class in $H^*_B(M)$ admits a symplectic Harmonic representative if and only if
 the manifold satisfies the transverse Hard Lefschetz property. \end{theorem}

%



\begin{theorem}\label{symplectic-ddelta}(\cite{Mer98}, \cite{Gui01}, \cite{He10}) Assume that $M$ is a compact odd dimensional symplectic manifold which satisfies the transverse Hard Lefschetz property, and which admits a connection
one form. Then on the space of basic forms, we have the following result.  \[ \text{im} d\cap \text{ker}\delta =\text{ker} d \cap \text{im} \delta=\text{im} d\delta.\]

\end{theorem}

Next, we present the primitive decomposition of the basic cohomology. We first define the basic version of the primitive cohomology as follows.

\begin{definition}\label{primitive1}Let $(M,\omega, \Omega)$ be a $2n+1$ dimensional symplectic manifold. For any $0\leq r \leq n$, the $r$-th primitive basic cohomology group, $PH_B^{r}(M,\R)$, is defined as follows.
 \[   PH_B^r (M,\R)= \text{ker}(L^{n-r+1}: H_B^r(M,\R)\rightarrow H_B^{2n-r+2}(M,\R)) .\]
\end{definition}





When the odd-dimensional symplectic manifold $M$ satisfies the transverse Hard Lefschetz property, the following primitive decomposition holds for basic De Rham cohomology.


\begin{theorem} (c.f. \cite{Yan96})\label{primitive-decomposition} Assume that $M$ has the transverse Hard Lefschetz property. Then
\begin{equation}   H_B^k(M,\R)=\bigoplus_r L^r PH_B^{k-2r}(M,\R).
\end{equation}\end{theorem}

The following result does not assume that $M$ has the transverse Hard Lefschetz property. Its proof is completely analogous to the case of even dimensional symplectic Hodge theory.  We refer to \cite{Yan96} for details.

\begin{lemma} \label{Yan's-lemma} Any primitive cohomology class in $PH_B^r(M,\R)$ is represented by a closed primitive basic form.
\end{lemma}

Finally, we collect here a few commutator relations which we will use later in this paper.

\begin{lemma} \label{commutator} \[ [d, \Lambda]=\delta,\,\,\, [\delta, L]=d, \,\,\,[d\delta, L]=0,\,\,\,[d\delta,\Lambda]=0.\]

\end{lemma}

\section{Review of contact and Sasakian geometry}\label{review-contact}

 Let $(M,\eta)$ be a co-oriented contact manifold with a contact one form $\eta$. We say that $(M,\eta)$ is $K$-contact if there is an endomorphism
 $\Phi: TM\rightarrow TM$ such that the following conditions are satisfied.
 \begin{itemize}
  \item[1)]  $\Phi^2=-Id+\xi\otimes \eta$, where $\xi$ is the Reeb vector field of $\eta$;
  \item [2)]  the contact one form $\eta$ is compatible with $\Phi$ in the sense that
  \[ d\eta(\Phi(X),\Phi(Y))=d\eta(X,Y)\]
  for all $X$ and $Y$, moreover, $d\eta(\Phi(X),X)>0$ for all non-zero $X \in \text{ker}\,\eta$;

  \item[3)] the Reeb field of $\eta$ is a Killing field with respect to the Riemannian metric defined
  by the formula
  \[ g(X,Y)=d\eta(\Phi(X),Y)+\eta(X)\eta(Y).\]

 \end{itemize}

Given a $K$-contact structure $(M,\eta,\Phi,g)$, one can define a metric cone
\[ (C(M), g_C)=(M\times \R_+, r^2g+dr^2),\] where $r$ is the radial coordinate. The $K$-contact structure $(M,\eta,\Phi)$ is called Sasakian if this metric cone is a K\"ahler manifold with K\"ahler form $\dfrac{1}{2} d(r^2\eta)$.

Let $(M,\eta)$ be a contact manifold with contact one form $\eta$ and a characteristic Reeb vector $\xi$. We note that the basic cohomology on $M$
 given in (\ref{basic}) in the context of odd dimensional symplectic geometry agrees with the usual basic cohomology with respect to the characteristic foliation on $M$. We need the following result from \cite[Sec. 7.2]{BG08}, which plays an important role in our work.

\begin{proposition}\label{exact-sequence}\begin{itemize}   \item[1)] On any $K$-contact manifold $(M,\eta)$, there is a long exact cohomology sequence
\begin{equation}\label{Gysin}  \cdots \rightarrow H^k_B(M,\R) \xrightarrow{i_*} H^k(M, \R)\xrightarrow{j_k} H^{k-1}_B(M,\R)\xrightarrow{\wedge[d\eta]} H^{k+1}_B(M,\R)\xrightarrow{i_*} \cdots,\end{equation} where
$i_*$ is the map induced by the inclusion, and $j_k$ is the map induced by $\iota_{\xi}$.

\item[2)] If $(M,\eta)$ is a compact $K$-contact manifold of dimension $2n+1$, then for any $r\geq 0$ the basic cohomology $H_B^r(M,\R)$ is finite dimensional, and for $r>2n$,
the basic cohomology $H_B^r(M,\R)=0$; moreover, for any $0\leq r\leq 2n$, there is a non-degenerate pairing
\[ H^r_B(M,\R)\otimes H^{2n-r}_B(M,\R)\rightarrow \R,\,\,\,([\alpha]_B,[\beta]_B)\mapsto \int_M\, \eta\wedge\alpha\wedge\beta. \]

\end{itemize}

\end{proposition}
On a compact Sasakian manifold $M$, the following Hard Lefschetz theorem is due to El Kacimi-Alaoui \cite{ka90}.

\begin{theorem}(\cite{ka90})\label{trans-Kahler} Let $(M,\eta,g)$ be a $2n+1$ dimensional compact Sasakian manifold with a contact one form $\eta$ and a Sasakian metric $g$. Then $M$ satisfies the transverse Hard Lefschetz property.
\end{theorem}

More recently,  Cappelletti-Montano, De Nicola, and Yudin \cite{CNY13} established a Hard Lefschetz theorem for the De Rham cohomology group of a compact Sasakian manifold.

\begin{theorem}(\cite{CNY13}) \label{HLP-sasakian} Let $(M,\eta,g)$ be a $2n+1$ dimensional compact Sasakian manifold with a contact one form $\eta$ and a Sasakian metric $g$, and let $\Pi: \Omega^*(M)\rightarrow \Omega_{har}^*(M)$ be the projection onto the space of Harmonic forms. Then for any $0\leq k\leq n$, the map
\[Lef_k: H^{k}(M,\R)\rightarrow H^{2n+1-k}(M,\R), [\beta]\mapsto [\eta\wedge (d\eta)^{n-k}\wedge \Pi \beta]\]
is an isomorphism. Moreover, for any $[\beta]\in H^k(M,\R)$, and for any closed basic primitive $k$-form $\beta'\in [\beta]$, $[\eta\wedge (d\eta)^{n-k}\wedge \beta']=Lef_k([\beta])$. In particular, the Lefschetz map $Lef_k$ does not depend on the choice of a compatible Sasakian metric.

\end{theorem}

This result motivates them to propose the following definition of the Hard Lefschetz property for a contact manifold.

\begin{definition}\label{HLP-contact} Let $(M,\eta)$ be a $2n+1$ dimensional compact contact manifold with a contact $1$-form $\eta$. For any $0\leq k\leq n$,
define the Lefschetz relation between the cohomology group $H^{k}(M,\R)$ and $H^{2n+1-k}(M,\R)$ to be
\begin{equation}\label{Lef-relation} \mathcal{R}_{Lef_k}=\{([\beta],[\eta\wedge L^{n-k}\beta])\,\vert \iota_{\xi}\beta=0, d\beta=0, L^{n-k+1}\beta=0\}.\end{equation}
If it is the graph of an isomorphism $Lef_k: H^{k}(M,\R)\rightarrow H^{2n+1-k}(M,\R)$ for any $0\leq k\leq n$, then
the contact manifold $(M,\eta)$ is said to have the hard Lefschetz property.

\end{definition}

\section{Hard Lefschetz theorem for K-contact manifolds}\label{main-result}

 Throughout this section, we assume $(M,\eta)$ to be a $2n+1$ dimensional compact $K$-contact manifold with a contact $1$-form $\eta$, and a Reeb vector field $\xi$. Set $\omega=d\eta$, and $\Omega=\eta\wedge\dfrac{\omega^n}{n!}$. Then $(M,\omega,\Omega)$ is an odd dimensional symplectic manifold in the sense of Definition \ref{odd-symplectic}. We will use extensively the machinery from odd dimensional symplectic Hodge theory as we explained in Section
\ref{odd-sym-hodge}.

\begin{lemma}\label{tech-lemma1}Let $(M,\eta)$ be a $2n+1$ dimensional compact $K$-contact manifold with a contact $1$-form $\eta$.  Assume that $M$ satisfies the transverse Hard Lefschetz property. Then for any $0\leq k\leq n$, the  map \[ i_*: H^k_B(M, \R)\rightarrow H^k(M,\R)\] is surjective; moreover, its image equals
\begin{equation} \label{image} \{i_*[\alpha]_B\,\vert\, \alpha \in \Omega^k_{bas}(M), d\alpha=0, \omega^{n-k+1}\wedge \alpha=0\}.\end{equation}
As a result, the restriction map $i_*: PH^k_{B}(M,\R)\rightarrow H^k(M,\R)$ is an isomorphism.
\end{lemma}

\begin{proof} Consider the long exact sequence (\ref{Gysin}). By assumption, $M$ satisfies the transverse Hard Lefschetz property. Thus  the map \[H^{i}_B(M,\R)\xrightarrow{\wedge[\omega]} H_B^{i+2}(M,\R)\] is injective for any
$ 0\leq i \leq n-1$. It then follows from the exactness of the sequence (\ref{Gysin}) that the map

\[ i_*: H^{k}_B(M,\R)\rightarrow H^{k}(M,\R)\] is surjective for any  $ 0\leq k\leq n$. This proves the first assertion in Lemma \ref{tech-lemma1}.

Since $M$ satisfies the transverse Hard Lefschetz property, by Theorem \ref{primitive-decomposition},
\[ H^{k}_B(M,\R)= PH^k_B(M)\oplus L H^{k-2}_B(M,\R).\]

It is clear from the exactness of the sequence (\ref{Gysin}) that
\[i_*\left(H^k_B(M,\R)\right)=i_*\left(PH^k_B(M,\R)\right),\,\,\,\text{ker} i_* \cap PH^k(M,\R)=0.\]
Therefore the restriction map $i_*:PH^k_B(M,\R)\rightarrow H^k(M,\R)$ is an isomorphism. Now applying Lemma \ref{Yan's-lemma}, it follows immediately that  $i_*\left( H^k_B(M,\R)\right)$ equals (\ref{image}). This completes the proof of Lemma \ref{tech-lemma1}.

\end{proof}

\begin{remark} The result proved in Lemma \ref{tech-lemma1} is known to hold for compact Sasakian manifolds, c.f. \cite[Prop. 7.4.13]{BG08}.
The traditional proof uses Riemannian Hodge theory associated to a compatible Sasakian metric.
\end{remark}

 We are ready to define the Lefschetz map on the cohomology groups. In \cite{CNY13}, such maps are introduced using Riemannian Hodge theory associated to a compatible Sasakian metric. In contrast, we define these maps here using the symplectic Hodge theory on the space of basic forms.

  For any $ 0\leq k\leq n$, define $Lef_k : H^{k}(M,\R)\rightarrow H^{2n+1-k}(M,\R)$ as follows.  For any cohomology class
$[\gamma] \in H^{k}(M,\R)$, by Lemma \ref{tech-lemma1} there exists a closed primitive basic $k$-form $\alpha \in \mathcal{P}_{bas}^k(M)$ such that $i_*[\alpha]_B=[\gamma]$. Observe that $d \left( \eta\wedge L^{n-k}\wedge \alpha\right)= L^{n-k+1} \alpha=0$. We define \begin{equation}\label{main-map} Lef_k[\gamma]= [\eta\wedge L^{n-k} \alpha].\end{equation}

\begin{lemma}\label{tech-lemma2} Assume that $M$ satisfies the transverse Hard Lefschetz property. Then the map (\ref{main-map}) does not depend on the choice of closed primitive basic forms.
\end{lemma}

\begin{proof} Suppose that there are two closed primitive basic $k$-forms $\alpha_1$ and $\alpha_2$ such that
$i_*[ \alpha_1 ]_B=i_*[\alpha_2]_B\in H^{k}(M,\R)$. It follows from the exactness of the sequence (\ref{Gysin}) that
$ [\alpha_1]_B=[\alpha_2]_B +L [\beta]_B$ for some
closed basic $(k-2)$-form $\beta$. Since $M$ satisfies the transverse Hard Lefschetz property, by Theorem
\ref{Mathieu's-theorem} one may well assume that $\beta$ is symplectic Harmonic.

Therefore,  $\alpha_1-\alpha_2 -L \beta$ is both $d$-exact and $\delta$-closed. By Theorem \ref{symplectic-ddelta}, the symplectic $d\delta$-lemma,
there exists a basic $k$-form $\varphi$ such that \begin{equation} \label{difference} \alpha_1-\alpha_2 -L \beta= d\delta \varphi \end{equation}

Lefschetz decompose $\beta$ and $\varphi$ as follows.
\[ \begin{split}  &\beta=\beta_{k-2}+L\beta_{k-4}+L^2\beta_{k-6}+\cdots \\
& \varphi=  \varphi_{k}+L\varphi_{k-2}+L^2\varphi_{k-4}+\cdots
\end{split} \] Here $\varphi_{k-i}\in \mathcal{P}_{bas}^{k-i}(M)$, $i=0, 2,\cdots$,  and $\beta_{k-i}\in \mathcal{P}_{bas}^{k-i}(M)$, $i=2,4,\cdots$.
Since $d\delta$ commutes with $L$, it follows from (\ref{difference}) that
\[ \alpha_1-\alpha_2=d\delta \varphi_k +L(\beta_{k-2}+d\delta \varphi_{k-2})+L^2(\beta_{k-4}+d\delta\varphi_{k-4})\cdots  .\]

Since $d\delta$ commutes with $\Lambda$, $d\delta$ maps primitive forms to primitive forms. It then follows from the uniqueness of the Lefschetz decomposition that
\[\alpha_1-\alpha_2=d\delta \varphi_k.\]

Observe that \[\begin{split} \eta\wedge \left(\omega^{n-k}\wedge (\alpha_1-\alpha_2)\right)&=
\eta\wedge \left(\omega^{n-k}\wedge d\delta\varphi_k\right)\\&=
-d\left( \eta\wedge \omega^{n-k}\wedge \delta\varphi_k\right)+ \left(L^{n-k+1} \delta\varphi_k\right).
\end{split}\]

Now using the commutator relation $[L, \delta]=-d$ repeatedly, it is clear that $L^{n-k+1} \delta\varphi_k$
must be $d$-exact, since $\varphi_k$ is a primitive $k$-form and so $L^{n-k+1}\varphi_{k}=0$. It follows immediately
that $\eta\wedge L^{n-k} (\alpha_1-\alpha_2)$ must be $d$-exact. This completes the proof of Lemma \ref{tech-lemma2}.

\end{proof}

\begin{theorem} \label{main-result1}Let $M$ be a $2n+1$ dimensional compact $K$-contact manifold with a contact one form $\eta$. Then it satisfies the Hard Lefschetz property if and only if it satisfies the transverse Hard Lefschetz property.
\end{theorem}

\begin{proof} {\bf Step 1.} \,Assume that $M$ satisfies the transverse Hard Lefschetz property. We show that $M$ satisfies the Hard Lefschetz property. Since $M$ is oriented and compact, in view of the Poincar\'e duality, it suffices to show that for any $0\leq k\leq n$, the map given in (\ref{main-map}) is injective.

Suppose that $Lef_k[\gamma]=[\eta\wedge L^{n-k} \alpha]=0$, where $\alpha \in \mathcal{P}_{bas}^k(M)$ such that $d\alpha=0$,  $i_*[\alpha]_B=[\gamma]$.
Since the group homomorphism $j_{2n+1-k}:H^{2n+1-k}(M,\R)\rightarrow H_B^{2n-k}(M,\R)$ is induced by $\iota_{\xi}$, it follows that
\[0=j_{2n+1-k}(0)=j_{2n+1-k}([\eta\wedge (L^{n-k} \alpha)])=[ L^{n-k} \alpha]_B.\]

Since $M$ has the transverse Hard Lefschetz property, $[\alpha]_B=0$. Thus $[\gamma]=i_*([\alpha]_B)=0$.

{\bf Step 2.}\, Assume that $M$ satisfies the Hard Lefschetz property. We show that for any $0\leq k\leq n$, the map
\begin{equation}\label{induction}  L^{n-k}: H^k_B(M)\rightarrow H_B^{2n-k}(M),\,\,\,[\alpha]_B\mapsto [\omega^{n-k}\wedge \alpha]_B\end{equation} is an isomorphism by induction on $k$. By Part 2) in Proposition \ref{exact-sequence}, it suffices to show that for any $0\leq k\leq n$, the map
(\ref{induction}) is injective.

By assumption, for any $0\leq k\leq n$,
\[ \mathcal{R}_{Lef_k}=\{([\beta],[\eta\wedge L^{n-k}\beta])\,\vert \iota_{\xi}\beta=0, d\beta=0, L^{n-k+1}\beta=0\}\]
is the graph of an isomorphism $Lef_k: H^k(M)\rightarrow H^{2n-k+1}(M)$.

For $0\leq k\leq n$, consider the map $i_*: PH^k_B(M,\R)\rightarrow H^k(M,\R)$. Since $\mathcal{R}_{Lef_k}$ is the graph of a map,  one sees that $i_*$ must be surjective. Furthermore, when $k=0,1$, for simple dimensional reasons,   $PH_B^k(M,\R)=H^k_B(M,\R)$ and that
the map $i_*: PH^k_B(M,\R)\rightarrow H^k(M,\R)$ is an isomorphism.

Now consider the long exact sequence (\ref{Gysin}) at stage $2n+1$. Since $H_B^{i}(M,\R)=0$ when $i\geq 2n+1$, we have that
\begin{equation}\label{Gysin-final-stage}\cdots \rightarrow 0 \xrightarrow{i_*} H^{2n+1}(M, \R)\xrightarrow{j_{2n+1}} H^{2n}_B(M,\R)\xrightarrow{\wedge[\omega]} 0\rightarrow  \cdots
\end{equation}

It follows that the map $j_{2n+1}: H^{2n+1}(M,\R)\rightarrow H_B^{2n}(M,\R)$ is an isomorphism.

Suppose that there is $[\alpha]_B\in PH^k_B(M,\R)$ such that  $L^{n-k}[\alpha]_B=0\in H^{2n-k}_B(M,\R)$. By Lemma \ref{Yan's-lemma}, we may assume that $\alpha$ is a closed primitive basic $k$-form. Then for any closed primitive basic $k$-form $\beta$,
\[ j_{2n+1}([\eta\wedge L^{n-k}\alpha \wedge \beta])= L^{n-k}[\alpha]_B\wedge [\beta]_B=0.\]

Since $j_{2n+1}$ is an isomorphism, it follows that $Lef_k(i_*[\alpha]_B) \cup i_*[\beta]_B=[\eta\wedge L^{n-k}\alpha \wedge \beta]=0$.  Since $\beta$ is arbitrarily chosen, by the Poincar\'e duality, we must have $Lef_k[i_*[\alpha])=0$. Since $Lef_k$ is an isomorphism, $i_*[\alpha]=0$. By the exactness of the sequence (\ref{Gysin}), $[\alpha]_B=L[\lambda]_B$ for some $[\lambda]_B\in H_B^{k-2}(M,\R)$.
For dimensional considerations, when $k=0,1$, we must have $[\alpha]_B=0$.
This proves that the map (\ref{induction}) is an isomorphism when $k=0,1$.

Assume that the map (\ref{induction}) is an isomorphism for any non-negative integer less than $k$.
 We first observe that the inductive hypothesis implies $H^k_B(M,\R)=PH^k_B(M,\R)+\text{im}\,L$. Indeed, by the inductive hypothesis, $L^{n-k+2}: H_B^{k-2}(M)\rightarrow H_B^{2n-k+2}(M)$ is an isomorphism. Therefore, for any $ [\varphi]_B\in H_B^{k}(M)$, \[L^{n-k+1}[\varphi]_B=L^{n-k+2}[\sigma]_B\] for some
$[\sigma]_B\in H_B^{k-2}(M)$. As a result, $L^{n-k+1}\left([\varphi]_B-L[\sigma]_B\right)=0$. This implies that
$[\varphi]_B-L[\sigma]_B\in PH^k_B(M)$ and so $[\varphi]_B\in PH_B^k(M,\R)+\text{im}\, L$.


Now suppose that $L^{n-k}([\alpha]_B+L[\sigma]_B)=0$, where $[\alpha]_B\in PH_B^{k}(M,\R)$ and $[\sigma]\in H_B^{k-2}(M,\R)$. Then we must have
$L^{n-k+1}([\alpha]_B+L[\sigma]_B)=L^{n-k+2}[\sigma]_B=0$. It follows from our inductive hypothesis again that
$[\sigma]_B=0$. As a result, $L^{n-k}[\alpha]_B=0$. By our previous work, we must have that $[\alpha]_B=L[\beta]_B$ for some $[\beta]_B\in H^{k-2}_B(M,\R)$.
Thus $L^{n-k+1}[\beta]_B=0$. By our inductive hypothesis again, we have that $[\beta]_B=0$ and so $[\alpha]_B=L[\beta]_B=0$.
This completes the proof that the map (\ref{induction}) is an isomorphism for any $0\leq k\leq n$.

\end{proof}

 The following result is an immediate consequence of Theorem \ref{main-result1}, which asserts that for a compact Sasakian manifold, the two existing versions of Hard Lefschetz theorems in the literature (c.f. \cite{ka90}, \cite{CNY13}) are mathematically equivalent to each other.

\begin{corollary} Assume that $M$ is a compact Sasakian manifold. Then the following two statements are equivalent to each other.

 \begin{itemize} \item [1)] $M$ satisfies the Hard Lefschetz property (as given in Definition \ref{HLP-contact}).
 \item [2)] $M$ satisfies the transverse Hard Lefschetz property. \end{itemize} \end{corollary}






\section{Simply-connected $K$-contact manifolds without Sasakian structures}\label{K-contact-examples}

It is well known that Boothby-Wang construction provides important examples of $K$-contact manifolds. In this section, we apply Theorem \ref{main-result1} to a Boothby-Wang fibration over a weakly Lefschetz symplectic manifold, and construct examples of simply-connected $K$-contact manifolds which do not support any Sasakian structures in any dimension $\geq 9$.

The notion of a weakly Lefschetz symplectic manifold was introduced in \cite{FMU04} and \cite{FMU07}.  A $2n$ dimensional symplectic manifold $(X,\sigma)$ is said to satisfy the $s$-Lefschetz property, where  $0\leq s\leq n-1$,  if for any $0\leq k\leq s$, the Lefschetz map
\[ L^{n-k}: H^k(X,\R)\rightarrow H^{2n-k}(X,\R),\,\,\,\, [\alpha]\mapsto [\omega^{n-k}\wedge \alpha] \]
is surjective.  In particular,  when $s=n-1$, we say that $(X,\sigma)$ satisfies the Hard Lefschetz property. The following result gives an useful criterion on when a $2n$ dimensional symplectic manifold is $s$-Lefschetz.

\begin{proposition}(\cite[Prop. 2.5]{FMU07})\label{criterion-s-lef} Let $(M,\omega)$ be a $2n$ dimensional symplectic manifold, and let $0\leq s\leq n-1$. Then  $(M,\omega)$ is $s$-Lefschetz if and only if for every $0\leq k\leq s$, any cohomology class in $H^{2n-k}(M,\R)$ has a harmonic representative.
\end{proposition}

Applying Proposition \ref{criterion-s-lef}, we prove a simple lemma on when a product symplectic manifold is $s$-Lefschetz.

\begin{lemma} \label{prod-s-lef} Let $(M_1,\omega_1)$ and $(M_2,\omega_2)$ be two symplectic manifold of dimension $2p$ and $2n-2p$ respectively, $0\leq p\leq n$, and let $(M,\omega)$ be the product symplectic manifold $(M_1\times M_2,\omega_1\times \omega_2)$. The following statements hold.
\begin{itemize}
\item [a)]  Let $\alpha_1$,$\alpha_2$ be harmonic forms on $(M_1,\omega_1)$ and $(M_2,\omega_2)$ respectively. Then $\alpha_1\wedge \alpha_2$ is a harmonic form on $(M,\omega)$.

\item [b)] If $(M_1,\omega_1)$ satisfies the $s$-Lefschetz property, $0\leq s\leq p-1$, and if $(M_2,\omega_2)$ satisfies the Hard Lefschetz property, then $(M,\omega)$ satisfies the $s$-Lefschetz property.

\end{itemize}

\end{lemma}

\begin{proof}  a) is an easy consequence of \cite[Lemma 1.4]{Yan96}. In view of Proposition \ref{criterion-s-lef}, to prove b) it suffices to show that for a fixed integer $0\leq k\leq s$, any cohomology class in $H^{2n-k}(M,\R)$ has a harmonic representative.  By the K\"unneth formular,  we have that
\[ H^{2n-k}(M,\R)= \bigoplus_{i+j=2n-k}H^{i}(M_1,\R)\otimes H^j(M_2,\R).\]
Let $i$ and $j$ be a pair of non-negative integers such that $i+j=2n-k$. Then we have that $i=2n-k-j\geq 2n-s$. Now let $[\alpha_1]\in H^i(M_1,\R)$, and let $[\alpha_2]\in H^j(M,\R)$. Since $(M_1,\omega_1)$ satisfies the $s$-Lefschetz property, and $(M_2,\omega_2)$ satisfies the Hard Lefschetz property,  we see that both $[\alpha_1]$ and $[\alpha_2]$ admit a harmonic representative on $M_1$ and $M_2$ respectively. By a), this proves that any cohomology class in $H^i(M_1,\R)\otimes H^j(M_2,\R)$ admits a harmonic representative. Since $i$ and $j$ are arbitrarily chosen, we conclude that any cohomology class in $H^{2n-k}(M,\R)$ has a harmonic representative. 
\end{proof}

 The following result will play an important role in our construction of simply-connected $K$-contact manifolds that do not admit any Sasakian structures.

\begin{theorem}(\cite[Prop. 5.2]{FMU07}) \label{example-weak-lef} Let $s \geq 2$ be an even integer. Then there is a simply-connected symplectic $(W_s,\sigma)$ of dimension $2(s+2)$ which is $s$-Lefschetz but not $(s+1)$-Lefschetz. Moreover, the symplectic form $\sigma$ is integral, and $b_{s+1}(W_s)=3$.
\end{theorem}

\begin{remark} By \cite[Theorem 4.2]{FMU07}, the symplectic form on  $M_s$ constructed in \cite[Prop. 5.1]{FMU07} can chosen to be integral. A careful reading of the proof of \cite[Prop.5.2]{FMU07} shows that the symplectic form on $W_s$ can also chosen to be integral; moreover, $b_{s+3}(W_s)=3$. Thus by the Poincar\'e dulaity, we have that $b_{s+1}(W_s)=3$.
\end{remark}

\begin{corollary}\label{example-weak-lef2} For any $n\geq 4$, there exists an $2n$ dimensional simply-connected compact symplectic manifold $(M,\omega)$, which is $2$-Lefschetz, and which satisfies the following properties.

\begin{itemize}
\item [a)] $[\omega]$ represents an integral cohomology class in $H^2(M,\Z)$;
\item [b)] $b_3(M)=3$.
\end{itemize} \end{corollary}

\begin{proof}  By Theorem \ref{example-weak-lef}, there exists an eight dimensional simply-connected symplectic manifold $(W,\sigma)$,  which is $2$-Lefschetz, and which has an integral symplectic form $\sigma$. Moreover, we have that $b_3(W)=3$. Now for any integer $n\geq 4$, let $(CP^{n-4},\omega_F)$ be the projective space equipped with a K\"ahler two form induced by the standard  Fubini-Study metric, and let  $(M,\omega)$ be the product symplectic manifold $(W\times CP^{n-4}, \sigma\times \omega_F)$.  Then by Lemma \ref{prod-s-lef}, $M$ is $2$-Lefschetz. Moreover, by construction $M$ is clearly simply-connected, and $\omega$ is integral.  An easy application of K\"unneth formula shows that $b_3(M)=3$. 
\end{proof}

Next we present a quick review of Boothby-Wang construction, and refer to \cite{B76} for more details.  A co-oriented contact structure on a $2n+1$ dimensional compact manifold $P$ is said to be regular if it is given as the kernel of a contact one form $\eta$, whose Reeb field $\xi$ generates a free effective $S^1$ action on $P$. Under this assumption,  $P$ is the total space of a principal circle bundle $ \pi: P\rightarrow M:=P/S^1$, and the base manifold $M$ is equipped with an integral symplectic form $\omega$ such that $\pi^* \omega =d\eta$.  Conversely, let $(M,\omega)$ be a compact symplectic manifold with an integral symplectic form $\omega$, and let $\pi:P\rightarrow M$ be the principal circle bundle over $M$ with Euler class $[\omega]$ and a connection one form $\eta$, such that $\omega=\pi^*d\eta$. Then $\eta$ is a contact one form on $P$ whose characteristic Reeb vector field generates the right translations of the structure group $S^1$ of this bundle.



\cite{Ha13} proves an useful result on when the total space of a Boothby-Wang fibration is simply-connected. Let $X$ be a compact and oriented manifold of dimension $m$. We say that $c \in H^2(X,\Z)$ is indivisible if the map
\[ c \cup : H^{m-2}(X,\Z)\rightarrow H^m(X,\Z)\] is surjective.

\begin{lemma} \label{boothby-wang-l1} (\cite[Lemma 16]{Ha13}) Let $\pi: P\rightarrow M$ be a Boothby-Wang fibration, and let $\omega$ be an integral symplectic form on $M$ which represents the Euler class of the Boothby-Wang fibration. Then $\pi_1(P)$ is simply-connected if and only if 
\begin{itemize}

\item [a)]$M$ is simply connected;

\item [b)] the Euler class $[\omega]$ is indivisible.
\end{itemize}
\end{lemma}

We are ready to prove the main result of Section \ref{K-contact-examples}.

\begin{theorem}\label{high-dim-example-2} For any $n\geq 4$, there exists a  simply-connected compact $K$-contact manifold $P$ of dimension $2n+1$, such that $b_3(P)=3$. In particular, $P$ does not support any Sasakian structure.
\end{theorem}

\begin{proof} By Corollary \ref{example-weak-lef2}, for any $n\geq 4$, there exists a simply-connected compact symplectic $(M,\omega)$ of dimension $2n$ which is $2$-Lefschetz. Moreover, the symplectic form $\omega$ is integral, and $b_{3}(M)=3$. without loss of generality, we may assume that $[\omega]$ is not an integer multiple of another integral cohomology class. Then by the Poincar\'e duality over integer coefficients, $[\omega]$ is an indivisible integral cohomology class. 

 Let $(P,\eta)$ be the Boothby-Wang firbation over $(M,\omega)$ whose Chern class is $[\omega]$. By Lemma \ref{boothby-wang-l1}, $P$ is simply-connected.  Consider the following portion of the Gysin sequence for the principal circle bundle $\pi:P\rightarrow M$.
\begin{equation}\label{Gysin-2}  \cdots H^{1}(M,\R)  \xrightarrow{\wedge[\omega]} H^{3}(M,\R) \xrightarrow{\pi^*} H^{3}(P, \R)\xrightarrow{\pi_*} H^{2}(M,\R)\xrightarrow{\wedge[\omega]} H^{4}(M,\R)\xrightarrow{\pi^*} \cdots,\end{equation} where
$\pi_*:H^*(P,\R)\rightarrow H^{*-1}(M,\R)$ is the map induced by integration along the fibre.

Since $M$ is $2$-Lefschetz,  the map  $H^{2}(M,\R)\xrightarrow{\wedge[\omega]} H^{4}(M,\R)$ must be injective. Since $M$ is simply-connected, $H^1(M,\R)=0$. As a result,  $b_{3}(P)=b_{3}(M)=3$. It follows from \cite[Theorem 7.4.11]{BG08} that $M$ can not support any Sasakian structure.
 \end{proof}




\medskip

\noindent
Yi Lin \\
Department of Mathematical Sciences \\
Georgia Southern University\\
203 Georgia Ave., Statesboro, GA, 30460 \\
{\em E\--mail}: yilin@georgiasouthern.edu

\noindent
\noindent

\end{document}